\documentclass[10pt]{amsart}
\usepackage{amsmath,amsfonts,amssymb}
\usepackage[colorlinks=true, linkcolor=blue,
  citecolor=blue, urlcolor=blue]{hyperref}
\def\newthm#1#2{\newtheorem{#1}[dummy]{#2}%
  \expandafter\def\csname#2\endcsname##1{\hyperref[#1:##1]{#2~\ref*{#1:##1}}}}
\newthm{thm}{Theorem}
\newthm{lemma}{Lemma}
\newthm{prop}{Proposition}
\newthm{cor}{Corollary}
\newthm{conj}{Conjecture}
\newthm{cons}{Consequence}
\newthm{fact}{Fact}
\newthm{obs}{Observation}
\newthm{guess}{Guess}
\theoremstyle{definition}
\newthm{defn}{Definition}
\newthm{remark}{Remark}
\newthm{example}{Example}
\newthm{question}{Question}
\newthm{notation}{Notation}

\newcommand{\Section}[1]{\hyperref[sec:#1]{Section~\ref*{sec:#1}}}
\newcommand{\Table}[1]{\hyperref[tab:#1]{Table~\ref*{tab:#1}}}
\newcommand{\eqn}[1]{\hyperref[eqn:#1]{(\ref*{eqn:#1})}}

\DeclareMathOperator{\GL}{GL}

\DeclareMathOperator{\Gr}{Gr}
\DeclareMathOperator{\Fl}{Fl}
\DeclareMathOperator{\LG}{LG}

\DeclareMathOperator{\OG}{OG}

\DeclareMathOperator{\Aut}{Aut}

\DeclareMathOperator{\QH}{QH}

\DeclareMathOperator{\pt}{pt}

\DeclareMathOperator{\codim}{codim}

\DeclareMathOperator{\ch}{\ch}

\usepackage{bm}

\newcommand{\ssm}{\smallsetminus}
\newcommand{\bP}{{\mathbb P}}

\newcommand{\C}{{\mathbb C}}

\newcommand{\Q}{{\mathbb Q}}
\newcommand{\Z}{{\mathbb Z}}

\newcommand{\cS}{{\mathcal S}}
\newcommand{\cT}{{\mathcal T}}

\newcommand{\gw}[2]{\langle #1 \rangle^{\mbox{}}_{#2}}

\newcommand{\al}{{\alpha}}
\newcommand{\be}{{\beta}}

\newcommand{\la}{{\lambda}}

\DeclareMathOperator{\ev}{ev}

\newcommand{\wh}{\widehat}
\newcommand{\wb}{\overline}
\newcommand{\ov}{\overline}

\newcommand{\ignore}[1]{}

\newcommand{\Mb}{\wb{\mathcal M}}

\newcommand{\noin}{\noindent}
\newcommand{\lex}{{\operatorname{lex}}}

\begin{document}

\title
[Positivity determines quantum cohomology]
{Positivity determines the quantum cohomology of Grassmannians}

\date{May 14, 2019}

\author{Anders~S.~Buch}
\address{Department of Mathematics, Rutgers University, 110
  Frelinghuysen Road, Piscataway, NJ 08854, USA}
\email{asbuch@math.rutgers.edu}

\author{Chengxi Wang}
\address{Department of Mathematics, Rutgers University, 110
  Frelinghuysen Road, Piscataway, NJ 08854, USA}
\email{cw674@math.rutgers.edu}


\thanks{The authors were supported in part by NSF grant DMS-1503662.}

\begin{abstract}
  We prove that if $X$ is a Grassmannian of type A, then the Schubert
  basis of the (small) quantum cohomology ring $\QH(X)$ is the only
  homogeneous deformation of the Schubert basis of the ordinary
  cohomology ring $H^*(X)$ that multiplies with non-negative structure
  constants.  This implies that the (three point, genus zero)
  Gromov-Witten invariants of $X$ are uniquely determined by Witten's
  presentation of $\QH(X)$ and the fact that they are non-negative.
  We conjecture that the same is true for any flag variety $X = G/P$
  of simply laced Lie type.  For the variety of complete flags in
  $\C^n$, this conjecture is equivalent to Fomin, Gelfand, and
  Postnikov's conjecture that the quantum Schubert polynomials of type
  A are uniquely determined by positivity properties.  Our proof for
  Grassmannians answers a question of Fulton.
\end{abstract}

\maketitle

\section{Introduction}

In this paper we give a proof of Bertram's structure theorems
\cite{bertram:quantum} for the (small) quantum cohomology ring of a
Grassmannian that uses only that this ring is a graded and associative
deformation of the singular cohomology ring
\cite{ruan.tian:mathematical, kontsevich.manin:gromov-witten},
satisfies Witten's presentation \cite{witten:verlinde}, and the fact
that the defining Gromov-Witten invariants are non-negative.  We
expect that similar results hold more generally, so we will start by
discussing the quantum cohomology of flag manifolds in general.

Let $X = G/P$ be a flag variety defined by a complex semisimple linear
algebraic group $G$ and a parabolic subgroup $P$.  Fix a maximal torus
$T$ and a Borel subgroup $B$ such that
$T \subset B \subset P \subset G$.  The opposite Borel subgroup
$B^- \subset G$ is defined by $B \cap B^- = T$.  Let $W = N_G(T)/T$ be
the Weyl group of $G$, $W_P = N_P(T)/T$ the Weyl group of $P$, and let
$W^P \subset W$ be the subset of minimal representatives of the cosets
in $W/W_P$.  Each element $w \in W^P$ determines a $B$-stable Schubert
variety $X_w = \ov{B w.P}$ and an (opposite) $B^-$-stable Schubert
variety $X^w = \ov{B^- w.P}$, such that
$\dim(X_w) = \codim(X^w,X) = \ell(w)$.  The Schubert classes $[X^w]$
form a basis of the cohomology ring $H^*(X,\Z)$.

Let $\Phi$ be the root system of $G$, with positive roots $\Phi^+$ and
simple roots $\Delta \subset \Phi^+$.  The parabolic subgroup $P$ is
determined by the subset
$\Delta_P = \{ \be \in \Delta \mid s_\be \in W_P \}$.  The group
$H_2(X,\Z)$ has a basis consisting of curve classes $[X_{s_\be}]$
indexed by the simple roots $\be \in \Delta \ssm \Delta_P$.  The
elements of $H_2(X,\Z)$ will be called \emph{degrees}, and a degree
$d = \sum_\be d_\be [X_{s_\be}]$ is \emph{effective}, written
$d \geq 0$, if $d_\be \geq 0$ for each $\be$.

Given an effective degree $d \in H_2(X,\Z)$, we let $\Mb_{0,3}(X,d)$
denote the Kontsevich moduli space of 3-pointed stable maps
$f : C \to X$ of genus zero and degree $d$.  See
\cite{fulton.pandharipande:notes} for a construction of this moduli
space.
The dimension is given by
\[
  \dim \Mb_{0,3}(X,d) = \dim X + \int_d c_1(T_X) \,.
\]
Let $\ev_i : \Mb_{0,3}(X,d) \to X$ be the evaluation map that sends a
stable map $f$ to the image of the $i$-th marked point in its domain,
for $1 \leq i \leq 3$.  Given $u, v, w \in W^P$ and $d \in H_2(X,Z)$
effective, the corresponding (3 point, genus zero) \emph{Gromov-Witten
  invariant} of $X$ is defined by
\[
  \gw{[X^u], [X^v], [X_w]}{d} = \int_{\Mb_{0,3}(X,d)}
  \ev_1^*[X^u] \cdot \ev_2^*[X^v] \cdot \ev_3^*[X_w] \,.
\]
This invariant is non-zero only if
$\ell(u) + \ell(v) = \ell(w) + \int_d c_1(T_X)$, in which case it is
the number of parametrized curves $\bP^1 \to X$ of degree $d$ that map
three points in $\bP^1$ to fixed general translates of the Schubert
varieties $X^u$, $X^v$, and $X_w$ (see
\cite{fulton.pandharipande:notes}).

To simplify the statements of our results and conjectures, we will
work with quantum cohomology over the field of rational numbers.  Let
$\Q[q] = \Q[q_\be : \be \in \Delta\ssm\Delta_P]$ denote a polynomial
ring in one variable $q_\be$ for each Schubert curve $X_{s_\be}$.
These variables are called \emph{deformation parameters}.  We equip
this ring with the grading defined by
$\deg(q_\be) = \int_{X_{s_\be}} c_1(T_X)$.  Given an effective degree
$d = \sum_\be d_\be [X_{s_\be}]$ in $H_2(X,\Z)$, set
$q^d = \prod q_\be^{d_\be}$.  The (small) \emph{quantum cohomology
  ring} $\QH(X)$ is a graded $\Q[q]$-algebra, which as a
$\Q[q]$-module is defined by $\QH(X) = H^*(X,\Q) \otimes_\Q \Q[q]$.
In other words, the Schubert classes $[X^w]$ for $w \in W^P$ form a
$\Q[q]$-basis of this ring.  The grading is given by
$\deg\, [X^w] = \ell(w)$, and the multiplicative structure is defined
by
\[
  [X^u] \star [X^v] = \sum_{w,d\geq 0} \gw{[X^u], [X^v], [X_w]}{d}\,
  q^d\, [X^w]
\]
where the sum is over all $w \in W^P$ and effective $d \in H_2(X,\Z)$.
It was proved by Ruan and Tian \cite{ruan.tian:mathematical} and by
Kontsevich and Manin \cite{kontsevich.manin:gromov-witten} that this
product is associative.

The definition of the quantum cohomology ring came from physics, see
\cite{witten:verlinde}.  In mathematics the quantum ring $\QH(X)$
provides an effective tool for computing the Gromov-Witten invariants
of $X$, as these invariants can be calculated if enough is known about
the structure of the ring $\QH(X)$.  For example, the computation of
Gromov-Witten invariants is reduced to combinatorics if one knows all
invariants $\gw{[X^u], [X^v], [X_w]}{d}$ for which $[X^u]$ belongs to
a set of generators of the cohomology ring $H^*(X,\Q)$.  Formulas for
Gromov-Witten invariants of this type are frequently called
\emph{quantum Pieri formulas} or \emph{quantum Chevalley formulas},
see \cite{bertram:quantum, fomin.gelfand.ea:quantum,
  ciocan-fontanine:quantum*2, fulton.woodward:quantum,
  kresch.tamvakis:quantum, kresch.tamvakis:quantum*1,
  mihalcea:equivariant*1, buch.kresch.ea:quantum,
  huang.li:equivariant} for examples.

It is in general not enough to know a presentation of the ring
$\QH(X)$ by generators and relations for the Gromov-Witten invariants
of $X$ to be determined (see \Example{lg24}).  However, a presentation
together with a \emph{quantum Giambelli formula} that expresses each
Schubert class $[X^w]$ as a polynomial in the chosen generators of
$\QH(X)$ is sufficient.  Examples of presentations and Giambelli
formulas in quantum cohomology can be found in \cite{witten:verlinde,
  siebert.tian:quantum, bertram:quantum, givental.kim:quantum,
  kim:quantum, ciocan-fontanine:quantum, kim:quantum*1,
  fomin.gelfand.ea:quantum, kresch.tamvakis:quantum,
  kresch.tamvakis:quantum*1, mihalcea:giambelli,
  buch.kresch.ea:giambelli*1, buch.kresch.ea:giambelli,
  anderson.chen:equivariant}.

It should be noted that quantum cohomology is not functorial, so there
is in general no ring homomorphism $\QH(G/P) \to \QH(G/B)$ that
extends the pullback map $H^*(G/P,\Q) \to H^*(G/B,\Q)$ along the
projection $G/B \to G/P$.  For this reason, the study of quantum
cohomology must be carried out separately for each flag variety $X$.
Relations between the Gromov-Witten invariants of $G/B$ and $G/P$ do
exist in the form of the Peterson comparison formula
\cite{woodward:d}, but it requires work to extract algebraic
properties of the quantum ring from this formula (see e.g.\
\cite{huang.li:equivariant}).

The first examples of structure theorems for the quantum cohomology
ring $\QH(X)$ were provided by Witten \cite{witten:verlinde} and
Bertram \cite{bertram:quantum} when $X = \Gr(m,n)$ is the Grassmannian
of $m$-planes in $\C^n$.  Bertram's results, a Pieri formula and a
Giambelli formula for $\QH(X)$, were proved by applying results from
intersection theory to the quot scheme compactification of the moduli
space ${\mathcal M}_{0,3}(X,d)$ of parametrized curves.  More
elementary proofs were later given in \cite{buch:quantum} by studying
the kernels and spans of the curves counted by Gromov-Witten
invariants.  The idea is that if a Gromov-Witten invariant of a
Grassmannian is non-zero, then the \emph{span} of any counted curve,
defined as the linear span of the $m$-planes given by points of the
curve, will produce a point in a related intersection of Schubert
varieties in a different Grassmannian, and this intersection can be
studied with classical Schubert calculus.  More generally, it was
shown in \cite{buch.kresch.ea:gromov-witten} that the map that sends a
curve to its kernel-span pair provides a bijection between the curves
counted by a Gromov-Witten invariant with the points in an
intersection of Schubert varieties in a two-step flag variety.
Further generalizations of this \emph{quantum equals classical}
phenomenon can be found in \cite{chaput.manivel.ea:quantum*1,
  buch.kresch.ea:quantum, buch.mihalcea:quantum,
  buch.chaput.ea:projected} and the references therein.

The quantum cohomology ring of the variety $X = \GL(n)/B$ of complete
flags in $\C^n$ is described by structure theorems of Fomin, Gelfand,
and Postnikov \cite{fomin.gelfand.ea:quantum}.  The main result of
\cite{fomin.gelfand.ea:quantum} states that the \emph{quantum Schubert
  polynomials} provide Giambelli formulas for the Schubert classes in
$\QH(X)$.  In addition to a known presentation of the ring $\QH(X)$
\cite{givental.kim:quantum, kim:quantum, ciocan-fontanine:quantum} and
a special case of the Giambelli formula
\cite{ciocan-fontanine:quantum}, the proof uses the geometric fact
that the structure constants of $\QH(X)$ are non-negative integers.
Moreover, it was conjectured in \cite{fomin.gelfand.ea:quantum} that
the quantum Schubert polynomials are uniquely determined by the
non-negativity of their structure constants, in addition to the
presentation, grading, and deformation properties of $\QH(X)$.
Positivity properties have received much attention in the study of
quantum cohomology \cite{mihalcea:positivity,
  anderson.chen:positivity, buch.kresch.ea:puzzle, buch:mutations},
but usually as a benchmark of our understanding of the combinatorial
aspects of quantum cohomology rather than a tool for proving other
results.

Fulton asked the related question, whether the quantum cohomology ring
of a Grassmannian is uniquely determined by positivity conditions
\cite{fulton:private}.  The immediate answer is no; it is always
possible to rescale the deformation parameter $q$ by a positive factor
$\al > 0$, or equivalently, multiply each Gromov-Witten invariant
$\gw{[X^u], [X^v], [X_w]}{d}$ by $\al^{-d}$.  However, this
modification will change the relations among the Schubert class
generators of $\QH(X)$.  In this paper we show that, for
Grassmannians, rescaling the Gromov-Witten invariants is the only
change to the quantum cohomology ring that preserves its formal
properties as well as positivity of its structure constants
(\Corollary{classify_qdeform}).  By adding the condition that the
quantum cohomology ring satisfies Witten's presentation
\cite{witten:verlinde}, this ring is uniquely determined, and
Bertram's structure theorems can be proved with purely combinatorial
methods.

Our results establish the Grassmannian case of the following
conjecture, which also generalizes the conjectured uniqueness of
quantum Schubert polynomials \cite{fomin.gelfand.ea:quantum}.

\begin{conj}\label{conj:unique}
  Let $X = G/P$ be any flag variety of simply laced Lie type.  Then
  the Schubert basis of $\QH(X)$ is the only homogeneous $\Q[q]$-basis
  that deforms the Schubert basis of $H^*(X,\Q)$ and multiplies with
  non-negative structure constants.
\end{conj}

In other words, if $\{ \tau_w \mid w \in W^P \}$ is any $\Q[q]$-basis
of $\QH(X)$ such that $\tau_w$ is homogeneous of degree $\ell(w)$,
$\tau_w - [X^w]$ belongs to the ideal $\langle q \rangle$ generated by
the deformation parameters $q_\be$, and each product
$\tau_u \star \tau_v$ is a non-negative linear combination of the
$\Q$-basis $\{ q^d \tau_w \}$, then we have $\tau_w = [X^w]$ for all
$w \in W^P$.

In addition to Grassmannians of type A, we have verified
\Conjecture{unique} when $X = \OG(n,2n)$ is a maximal orthogonal
Grassmannian of type D$_n$, with $n \leq 6$, and when $X = \OG(1,2n)$
is any quadric hypersurface of even dimension.  The condition that the
root system of $G$ is simply laced is necessary, since the conjecture
is false for the Lagrangian Grassmannian $\LG(2,4)$ of type C$_2$, a
3-dimensional quadric.  These examples are treated
in \Section{qdeform}.  We have also obtained computational
verification of \Conjecture{unique} for all partial flag varieties
$\GL(n)/P$ of dimension at most 10, except $\Fl(1,2,3;\C^5)$,
$\Fl(1,2,4;\C^5)$, and $\Fl(1,2,3,4;\C^5)$.  We note that, since the
number of monomials in $\Q[q]$ grows very fast with respect to total
degree when $X$ has large Picard rank, so does the computational
complexity of this problem.

Our proof of \Conjecture{unique} for Grassmannians of type A uses that
quantum multiplication with the top Chern class of a tautological
vector bundle maps any Schubert class to a power of $q$ times a
different Schubert class.  This follows from Bertram's quantum Pieri
formula \cite{bertram:quantum} and is a special case of the Seidel
representation of the group $\pi_1(\Aut(X))$ on
$\QH(X)/\langle q-1 \rangle$, which was introduced in \cite{seidel:1}
and computed explicitly for flag varieties in
\cite{chaput.manivel.ea:affine}.  Here $\langle q-1 \rangle$ is the
ideal generated by all differences $q_\be - 1$.  For any flag variety
$X$, the action of this representation is given by quantum
multiplication by certain Schubert classes $[X^u]$ of finite order in
the group of units $(\QH(X)/\langle q-1 \rangle)^\times$.  We show
that basis elements of finite order in any non-negative quantum
deformation must behave in a similar way (\Lemma{basis2basis}), which
gives a way to compare different basis elements.  For the Grassmannian
$X = \Gr(m,n)$ we then observe that the sum of the degrees of the
Schubert classes in any Seidel orbit is always the same number, namely
$\frac{1}{2} mn(n-m)$ (\Lemma{seidelsum}).  This is a special
phenomenon for Grassmannians of type A; a counterexample for the
variety of complete flags in $\C^6$ is provided in
\Example{counterexample}.  The Grassmannian case of
\Conjecture{unique} is proved by combining these observations.

The quantum cohomology ring of a flag variety is generalized by the
quantum $K$-theory ring \cite{givental:wdvv, lee:quantum*1}, which has
received much attention in recent years, see e.g.\
\cite{buch.mihalcea:quantum, buch.chaput.ea:finiteness,
  gorbounov.korff:quantum, iritani.milanov.ea:reconstruction,
  buch.chaput.ea:chevalley, lam.li.ea:conjectural, kato:frobenius,
  kato:loop, anderson.chen.ea:quantum, buch.chung.ea:euler}.  In
contrast to quantum cohomology, structure theorems for the quantum
$K$-theory ring are known only for cominuscule spaces
\cite{buch.mihalcea:quantum, gorbounov.korff:quantum,
  buch.chaput.ea:chevalley}, while conjectures exist in other cases
\cite{lenart.maeno:quantum, lenart.postnikov:affine}.  It would be
very interesting to find a way to uniquely describe the quantum
$K$-theory ring in terms of combinatorial properties.  The same
question can be asked about the equivariant quantum cohomology ring
\cite{kim:equivariant, kim:quantum*1, mihalcea:equivariant*1,
  mihalcea:positivity, anderson.chen:positivity}.

Another interesting open problem is to give a combinatorial proof that
the quantum cohomology ring of a Grassmannian, as defined by Bertram's
structure theorems \cite{bertram:quantum}, has non-negative structure
constants (see \cite{bertram.ciocan-fontanine.ea:quantum}).  While
positive formulas for these Gromov-Witten invariants do exist
\cite{buch.kresch.ea:puzzle, buch:mutations}, their proofs utilize the
quantum equals classical theorem \cite{buch.kresch.ea:gromov-witten,
  buch.mihalcea:quantum} rather than studying the quantum cohomology
ring directly.  The methods presented in this paper do not solve this
problem; we simply prove that any quantum deformation of a
Grassmannian with non-negative structure constants must obey Bertram's
structure theorems.

This paper is organized as follows.  In \Section{qdeform} we formulate
the notion of a \emph{quantum deformation}, which captures the
properties of the quantum cohomology ring that are automatic from its
definition.  We also provide some initial examples of
\Conjecture{unique}.  In \Section{grassmannians} we prove that any
non-negative quantum deformation of a Grassmannian is obtained from
the quantum cohomology ring by rescaling the defining Gromov-Witten
invariants.  This section can also be read as a combinatorial
introduction to the quantum cohomology of Grassmannians.  We thank
W.~Fulton for his inspiring question.

\section{Quantum deformations}\label{sec:qdeform}

\subsection{Quantum deformations}

The (small) quantum cohomology ring $\QH(X)$ of a flag variety
$X = G/P$ is a quantum deformation of the singular cohomology ring
$H^*(X,\Q)$ in the following sense.

\begin{defn}\label{defn:qdeform}
  A \emph{quantum deformation} of $H^*(X,\Q)$ is a graded associative
  $\Q[q]$-algebra $\QH$ with unit, together with a $\Q[q]$-basis
  $\{ \tau_w : w \in W^P \}$ of $\QH$ indexed by $W^P$, such that
  each basis element $\tau_w$ is homogeneous of degree $\ell(w)$,
  and the assignments $\tau_w \mapsto [X^w]$ and $q_\be \mapsto 0$
  define a ring homomorphism $\QH \to H^*(X,\Q)$.
\end{defn}

We will often denote a quantum deformation $(\QH, \{\tau_w\})$ simply
by $\QH$, and we will occasionally misuse terminology and call it a
quantum deformation of $X$.  Notice that
$\QH/\langle q \rangle \cong H^*(X,\Q)$, where
$\langle q \rangle \subset \QH$ is the ideal generated by the
deformation parameters $q_\be$.  In addition, the set
$\cT = \{ q^d \tau_w \mid w \in W^P, d \geq 0 \}$ is a basis of $\QH$
as a vector space over $\Q$.  The structure constants of $\QH$ with
respect to this basis are the numbers $N^{w,d}_{u,v} \in \Q$, indexed
by $u,v,w \in W^P$ and $d \geq 0$, defined by the identity
\[
  \tau_u \tau_v = \sum_{w,d} N^{w,d}_{u,v}\, q^d\, \tau_w \,.
\]
We will say that the quantum deformation $\QH$ is \emph{non-negative}
if $N^{w,d}_{u,v} \geq 0$ for all $u,v,w \in W^P$ and $d \geq 0$.

The quantum deformation of main interest is the quantum cohomology
ring $\QH(X)$, equipped with its basis of Schubert classes.  Since the
Gromov-Witten invariants of $X$ are enumerative
\cite{fulton.pandharipande:notes}, we know a priori that $\QH(X)$ is a
non-negative quantum deformation.  It is natural to ask which
non-negative quantum deformations exist.

One way to produce new quantum deformations is to replace the Schubert
basis of $\QH(X)$ with any $\Q[q]$-basis $\{ \tau_w : w \in W^P \}$
for which each element $\tau_w$ is a homogeneous deformation of the
Schubert class $[X^w]$, i.e.\
$\tau_w - [X^w] \in \langle q \rangle$.  Such deformations will be
called \emph{change-of-basis} quantum deformations of $H^*(X,\Q)$.

We remark that historically, a presentation of the quantum ring
$\QH(X)$ was known early on for many flag varieties
\cite{witten:verlinde, givental.kim:quantum, ciocan-fontanine:quantum,
  kim:quantum*1}.  For example, Witten's paper \cite{witten:verlinde}
contains a presentation of $\QH(X)$ when $X$ is a Grassmannian (see
also \cite{siebert.tian:quantum}).  When a presentation is known, a
determination of the quantum ring $\QH(X)$ and the associated
Gromov-Witten invariants of $X$ is equivalent to finding the correct
non-negative change-of-basis quantum deformation of $X$, so it is also
natural to ask for a classification of these deformations.
\Conjecture{unique} states that, if $X$ is of simply laced Lie type,
then $\QH(X)$ is the only non-negative change-of-basis quantum
deformation up to isomorphism.\footnote{An (iso)morphism
  $(\QH,\{\tau_w\}) \to (\QH',\{\tau'_w\})$ of quantum deformations is
  a $\Q[q]$-algebra homomorphism that maps $\tau_w$ to $\tau'_w$ for
  each $w \in W^P$.}

\subsection{General properties}

Let $(\QH, \{\tau_w\})$ be a quantum deformation of $H^*(X,\Q)$ and
let $\phi : \QH \to H^*(X,\Q)$ be the ring homomorphism of
\Definition{qdeform}.  Assume that $\phi(\eta) = [X^u]$ for some
$\eta \in \QH$ and $u \in W^P$.  Then
$\eta - \tau_u \in \langle q \rangle$.  If $\eta$ is homogeneous with
$\deg(\eta) < \deg(q_\be)$ for all $\be \in \Delta \ssm \Delta_P$,
then this implies that $\eta = \tau_u$.  For example we have
$\tau_e = 1$, where $e \in W$ is the identity element.

\begin{lemma}\label{lemma:basis2basis}
  Let $(\QH, \{\tau_w\})$ be any non-negative quantum deformation of
  $H^*(X,\Q)$ and let $u_1,u_2,\dots,u_\ell \in W^P$.  Assume that the
  product $\tau_{u_\ell} \tau_{u_{\ell-1}} \cdots \tau_{u_1}$ is a
  non-zero multiple of a monomial from $\Q[q]$, that is,
  $\tau_{u_\ell} \tau_{u_{\ell-1}} \cdots \tau_{u_1} = c\, q^d$ for
  some rational number $c > 0$ and effective degree $d$.
  \textnormal{(1)} For every $v \in W^P$, the product
  $\tau_{u_1} \tau_v$ is a positive multiple of a single element from
  $\cT = \{ q^d \tau_w \}$.  \textnormal{(2)} If
  $[X^{u_1}] \cdot [X^v] = [X^w]$ holds in $H^*(X,\Q)$, then
  $\tau_{u_1} \tau_v = \tau_w$ holds in $\QH$.
\end{lemma}
\begin{proof}
  If part (1) is false, then since
  $\tau_{u_\ell} \cdots \tau_{u_1} \tau_v = c\, q^d \tau_v$, we can
  choose $p > 0$ such that the expansion of
  $\tau_{u_p} \cdots \tau_{u_1} \tau_v$ involves more than one basis
  element from $\cT$, whereas
  $\tau_{u_{p+1}} \tau_{u_p} \cdots \tau_{u_1} \tau_v$ is a multiple
  of a single basis element $q^e \tau_w$.  Write
  $\tau_{u_p} \cdots \tau_{u_1} \tau_v = b_1\, q^{d_1} \tau_{v_1} +
  \dots + b_r\, q^{d_r} \tau_{v_r}$ where $r \geq 2$ and $b_i > 0$ for
  $1 \leq i \leq r$.  Since $\tau_{u_{p+1}} q^{d_i} \tau_{v_i}$ is a
  non-negative linear combination of $\cT$ for each $i$, and since
  $\tau_{u_{p+1}} \sum b_i\, q^{d_i} \tau_{v_i}$ is a multiple of
  $q^e \tau_w$, we deduce that $\tau_{u_{p+1}}\, q^{d_i} \tau_{v_i}$
  is a multiple of $q^e \tau_w$ for each $i$.  This contradicts that
  multiplication by $\tau_{u_{p+1}}$ is an injective operation on
  $\QH$.  Finally, part (2) follows from part (1).
\end{proof}

\subsection{Examples}

We finish this section by examining \Conjecture{unique} for a
selection of cominuscule spaces $X = G/P$.  The quantum cohomology
rings these varieties contain only one deformation parameter which we
denote by $q$.

\begin{example}\label{example:evenquadric}
  Let $X = \OG(1,2n)$ be the quadric hypersurface of dimension $2n-2$.
  A description of the Schubert classes and quantum cohomology ring of
  this variety can be found in \cite{chaput.manivel.ea:quantum*1,
    buch.kresch.ea:quantum}.  Let $\{\tau_w : w \in W^P\}$ be any
  homogeneous deformation of the Schubert basis of $\QH(X)$ that
  multiplies with non-negative structure constants.  Since
  $\deg(q) = \dim(X) = 2n-2$, it follows that $\tau_w = [X^w]$ for all
  $w \in W^P$, except for the element $w_0^P$ representing a point.
  Let $u \in W^P$ have length $n-1$, and choose $u^\vee \in W^P$ such
  that $[X^u] \cdot [X^{u^\vee}] = [X^{w_0^P}]$ in $H^*(X;\Q)$.  It
  follows from \cite[Thm.~1]{chaput.manivel.ea:affine} or
  \cite[Thm.~3.4]{buch.kresch.ea:quantum} that
  $\tau_u^4 = [X^u]^4 = q^2$, so we deduce from \Lemma{basis2basis}
  that
  $\tau_{w_0^P} = \tau_u \star \tau_{u^\vee} = [X^u] \star
  [X^{u^\vee}] = [X^{w_0^P}]$.  This shows that \Conjecture{unique}
  holds for $X$.  Similar arguments can be used to show that any
  non-negative quantum deformation of $X$ is obtained from $\QH(X)$ by
  rescaling the defining Gromov-Witten invariants.
\end{example}

\begin{example}\label{example:maxog6}
  Let $X = \OG(n,2n)$ be the maximal orthogonal Grassmannian of type
  D$_n$.  The Schubert classes and quantum cohomology of this space
  are described in \cite{kresch.tamvakis:quantum*1}.  We have
  $\dim(X) = \binom{n}{2}$ and $\deg(q) = 2n-2$.  The elements of
  $W^P$ can be identified with strict partitions
  $\la = (\la_1 > \la_2 > \dots > \la_\ell > 0)$ with
  $\la_1 \leq n-1$, so that $\codim(X^\la,X) = |\la| = \sum \la_i$.
  Let $\{\tau_\la\}$ be any homogeneous deformation of the Schubert
  basis of $\QH(X)$ that multiplies with non-negative structure
  constants.  Then we have $\tau_\la = [X^\la]$ for $|\la| < 2n-2$.
  The classical Pieri formula \cite{hiller.boe:pieri} implies that
  $[X^{n-1}] \cdot [X^\la] = [X^{n-1,\la}]$ holds in $H^*(X,\Q)$
  whenever $\la_1 < n-1$, and \cite[Cor.~5]{kresch.tamvakis:quantum*1}
  gives $\tau_{n-1}^2 = [X^{n-1}]^2 = q$.  \Lemma{basis2basis}
  therefore implies that
  $\tau_{n-1,\la} = \tau_{n-1} \star \tau_\la = [X^{n-1}] \star
  [X^\la] = [X^{n-1,\la}]$ for all strict partitions $\la$ such that
  $\la_1 < n-1$ and $|\la| < 2n-2$.  This proves \Conjecture{unique}
  when $n \leq 5$, and for $n = 6$ we obtain $\tau_\la = [X^\la]$ for
  all strict partitions except $\nu = (4,3,2,1)$ and $(5,\nu)$.  We
  must have $\tau_\nu = [X^\nu] + cq$ and
  $\tau_{5,\nu} = \tau_5 \star \tau_\nu = [X^{5,\nu}] + cq \tau_{5}$
  for some constant $c \in \Q$.  Since
  $\tau_1 \star \tau_{4,3,2} = [X^1] \star [X^{4,3,2}] = [X^{5,3,2}] +
  [X^\nu] = \tau_{5,3,2} + \tau_\nu - cq$, it follows that $c \leq 0$.
  Since
  $\tau_1 \star \tau_\nu = [X^1] \star [X^\nu] + c q [X^1] =
  [X^{5,3,2,1}] + cq [X^1] = \tau_{5,3,2,1} + cq \tau_1$, it follows
  that $c \geq 0$.  This shows that \Conjecture{unique} also holds for
  $\OG(6,12)$.
\end{example}

The following example shows that \Conjecture{unique} does not hold for
arbitrary flag varieties.  It would be interesting to know if the
non-negative quantum deformations of any flag variety correspond to
points in a rational polyhedral cone.

\begin{example}\label{example:lg24}
  Let $X = \LG(2,4)$ be the Lagrangian Grassmannian of type C$_2$.
  Then $H^*(X,\Q)$ has a basis of Schubert classes
  $\{[X^0]=1, [X^1], [X^2], [X^3] = [\pt]\}$, where
  $[X^p] \in H^{2p}(X)$.  The multiplicative structure of $H^*(X,\Q)$
  is determined by $[X^1]^2 = 2[X^2]$ and $[X^1]\cdot[X^2] = [X^3]$.
  One can check that any quantum deformation
  $(\QH,\{\tau_0=1, \tau_1,\tau_2,\tau_3\})$ of $H^*(X,\Q)$ is given
  by
  \[
    \begin{split}
      \tau_1^2 &\ = \ 2\, \tau_2 \\
      \tau_1\, \tau_2 &\ = \ \tau_3 + (2a-b)\, q \\
      \tau_1\, \tau_3 &\ = \ b\, q\, \tau_1 \\
      \tau_2^2 &\ = \ a\, q\, \tau_1 \\
      \tau_2\, \tau_3 &\ = \ b\, q\, \tau_2 \\
      \tau_3^2 &\ = \ (2b-2a)\, q\, \tau_3 + (2ab-b^2)\, q^2
    \end{split}
  \]
  where $a, b \in \Q$ are constants.  It is non-negative if and only
  if $a \leq b \leq 2a$, and it is a change-of-basis deformation if
  and only if $a=1$.  In particular, we obtain a non-negative
  change-of-basis deformation of $H^*(X,\Q)$ whenever $a=1$ and
  $1 \leq b \leq 2$.  The quantum cohomology ring $\QH(X)$ corresponds
  to $a=b=1$ \cite{kresch.tamvakis:quantum}.
\end{example}

\section{Grassmannians}\label{sec:grassmannians}

Let $X = \Gr(m,n)$ be the Grassmannian of $m$-dimensional vector
subspaces in $\C^n$ and set $k=n-m$.  In this section we classify the
non-negative quantum deformations of the cohomology ring $H^*(X,\Q)$.
We start by recalling some basic definitions and facts about symmetric
functions.

\subsection{Partitions}

We will identify any \emph{partition}
$\la = (\la_1 \geq \la_2 \geq \dots \geq \la_\ell > 0)$ with its
\emph{Young diagram} of boxes, which has $\la_i$ boxes in row $i$.
Rows are numbered from top to bottom and are left aligned.  The
\emph{length} $\ell(\la)$ is the number of non-empty rows, and the
\emph{weight} $|\la| = \sum \la_i$ is the total number of boxes.  We
write $\la_i = 0$ for $i > \ell(\la)$.  The \emph{conjugate} Young
diagram $\la'$ is obtained by mirroring $\la$ in the diagonal, so that
rows and columns are interchanged.

Given an additional partition $\mu$, we write $\mu \subset \la$ if and
only if $\mu_i \leq \la_i$ for each $i$.  In this case $\la/\mu$
denotes the \emph{skew diagram} of boxes in $\la$ that are not in
$\mu$.  This skew diagram is a \emph{horizontal strip} if it has at
most one box in each column, and a \emph{vertical strip} if it has at
most one box in each row.  The \emph{outer rim} of a $\la$ is the skew
diagram of boxes in $\la$ that are not strictly north and strictly
west of other boxes in $\la$.  In other words, the box in row $i$ and
column $j$ belongs to the outer rim of $\la$ if and only if
$\la_{i+1} \leq j \leq \la_i$.

For integers $r, c \geq 0$ we let $(c^r) = (c, c, \dots, c)$ denote
the rectangular Young diagram with $r$ rows and $c$ columns.  Since
the Schubert varieties in $X = \Gr(m,n)$ can be indexed by partitions
$\la$ contained in the rectangle $(k^m)$ (see \Section{grass_cohom}),
we also denote this rectangle as $m \times k$.

Recall also that the \emph{lexicographic order} on integer sequences
is defined by $\mu <_\lex \la$ if and only if for some index $p$ we
have $\mu_i = \la_i$ for $1 \leq i < p$ whereas $\mu_p < \la_p$.  The
sequence $\mu$ is smaller than $\la$ in the \emph{degree-lexicographic
  order} if and only if $|\mu| < |\la|$, or $|\mu| = |\la|$ and
$\mu <_\lex \mu$.

\subsection{Symmetric functions}

Given a sequence $h = (h_1, h_2, \ldots)$ of elements in a commutative
ring $R$ and a partition
$\la = (\la_1 \geq \la_2 \geq \dots \geq \la_\ell \geq 0)$, define
\[
  \Delta_\la(h) = \det( h_{\la_i+j-i} )_{\ell \times \ell} \,.
\]
Here we set $h_0 = 1$ and $h_i = 0$ for $i < 0$.  Let
$e = (e_1, e_2, \ldots)$ be the dual sequence defined by
$e_p = \Delta_{(1^p)}(h) = \det(h_{1+j-i})_{p \times p}$.
Equivalently, we can define $e_p$ recursively by the formula
\[
  e_p = \sum_{i=1}^p (-1)^{i-1} h_i e_{p-i} \,.
\]
We need the following facts about Schur polynomials.  Proofs can be
found in e.g. \cite[Ch.~I]{macdonald:symmetric*1} or \cite[\S 2.2, \S
6.1]{fulton:young}.

\begin{fact}\label{fact:pieri}
  We have $e_p\, \Delta_\la(h) = \sum_\mu \Delta_\mu(h)$,
  where the sum is over all partitions $\mu$ obtained by adding a
  vertical strip of $p$ boxes to $\la$.
\end{fact}

\begin{fact}\label{fact:conjugate}
  We have $\Delta_\la(e) = \Delta_{\la'}(h)$, where $\la'$ is the
  conjugate partition of $\la$.
\end{fact}

Define the polynomial ring $S = \Q[c_1,\dots,c_m]$, with the grading
defined by $\deg(c_i) = i$.  Let
$c = (c_1, c_2, \dots c_m, 0, 0, \dots)$ be the sequence of its
variables, with $c_i = 0$ for $i > m$.  Let
$\sigma = (\sigma_1, \sigma_2, \dots)$ be the dual sequence defined by
$\sigma_p = \Delta_{(1^p)}(c) = \det(c_{1+j-i})_{p \times p}$.  For
each partition $\la$ we set $\sigma_\la = \Delta_\la(\sigma)$.  This
is a homogeneous element of $S$ of degree $|\la|$.

\begin{lemma}\label{lemma:Sbasis}
  The set $\{ \sigma_\la \mid \ell(\la) \leq m \}$ is a $\Q$-basis of
  $S$, and we have $\sigma_\la = 0$ whenever $\ell(\la) > m$.
\end{lemma}
\begin{proof}
  By \Fact{conjugate} we have $\sigma_\la = \Delta_{\la'}(c)$.  If
  $\ell(\la) > m$, then all entries in the top row of the determinant
  defining $\Delta_{\la'}(c)$ are zero.  Otherwise the initial term of
  $\Delta_{\la'}(c)$ is
  $c_1^{\la_1-\la_2} c_2^{\la_2-\la_3} \cdots c_m^{\la_m}$, where the
  monomials $c_1^{p_1} c_2^{p_2} \dots c_m^{p_m}$ of $S$ are ordered
  according to the lexicographic order on the exponent vectors
  $(p_m, p_{m-1}, \dots, p_1)$.  The lemma follows from this.
\end{proof}

\begin{lemma}\label{lemma:Hbasis}
  The set $\{ \sigma_\la \mid \la \subset m \times k \}$ is a
  $\Q$-basis of the ring
  $S/\langle \sigma_{k+1}, \dots, \sigma_n \rangle$.
\end{lemma}
\begin{proof}
  By \Lemma{Sbasis} it is enough to show that
  $I = \langle \sigma_{k+1}, \sigma_{k+2}, \dots, \sigma_n \rangle$
  has a basis consisting of all elements $\sigma_\la$ for which
  $\ell(\la) \leq m$ and $\la_1 > k$.  \Fact{pieri} implies that $I$
  is contained in the span of these basis elements.  By using the
  identity $\sigma_p = \sum_{i=1}^m (-1)^{i-1} c_i \sigma_{p-i}$, it
  follows by induction on $p$ that $\sigma_p \in I$ for all $p > k$.
  Therefore all elements in the top row of the determinant defining
  $\sigma_\la$ are contained in $I$ when $\la_1 > k$.
\end{proof}

\begin{cor}\label{cor:nzd}
  The element $\sigma_n$ is a non-zero divisor in
  $S/\langle \sigma_{k+1}, \dots, \sigma_{n-1} \rangle$.
\end{cor}
\begin{proof}
  Set
  $J = \langle \sigma_{k+1}, \dots, \sigma_{n-1} \rangle \subset S$
  and $K = \{ f \in S \mid \sigma_n f \in J \}$.  We must show that
  $K = J$.  Otherwise choose a prime ideal $P \subset S$ such that
  $(K/J)_P \neq 0$.  Then $P$ contains
  $I = \langle \sigma_{k+1}, \sigma_{k+2}, \dots, \sigma_n \rangle$,
  and $S_P/I_P$ has Krull dimension zero by \Lemma{Hbasis}.  It
  follows that the generators of $I$ form a regular sequence in $S_P$
  \cite[Thm.~17.4]{matsumura:commutative*1},
  which contradicts that $\sigma_n$ is a zero-divisor in $S_P/J_P$.
\end{proof}

\subsection{Cohomology of Grassmannians}\label{sec:grass_cohom}

We summarize the main facts about the cohomology of Grassmannians.
More details and proofs can be found in e.g.\ \cite[\S
9.4]{fulton:young} or \cite[\S 3]{manivel:symmetric}.

Each partition $\la = (\la_1 \geq \dots \geq \la_m)$ contained in the
rectangle $m \times k$ defines a Schubert variety in $X = \Gr(m,n)$
given by
\[
  X^\la = \{ V \in X \mid \dim(V \cap F^{k+i-\la_i}) \geq i
  ~\forall 1 \leq i \leq m \} \,,
\]
where $F^{k+i-\la_i}$ denotes the span of the last $k+i-\la_i$
coordinate basis vectors in $\C^n$.  We have
$\codim(X^\la,X) = |\la|$.  With the notation of the introduction we
have $X = G/P$, where $G = \GL(n,\C)$ and $P$ is the stabilizer of the
span of the first $m$ basis vectors in $\C^n$.  If we use the maximal
torus $T$ of diagonal matrices and the Borel subgroup $B$ of upper
triangular matrices, and identify the Weyl group $W$ with the subgroup
of permutation matrices in $G$, then the set $W^P$ consists of all
permutations $w = (w_1,w_2,\dots,w_n)$ with descent only at position
$m$, i.e.\ $w_i < w_{i+1}$ for $i \neq m$.  We then have $X^w = X^\la$
where $\la = (w_m-m, \dots, w_2-2, w_1-1)$.

The Schubert classes $[X^\la]$ form a basis of $H^*(X,\Q)$ as a vector
space over $\Q$.  The special Schubert classes $[X^{(1^p)}]$ are the
Chern classes of the dual of the tautological subbundle
$\cS \subset X \times \C^n$ of rank $m$.  The classical Pieri formula
states that multiplication with these classes in $H^*(X,\Q)$ is
determined by
\begin{equation}\label{eqn:pieri}
  [X^{(1^p)}] \cdot [X^\la] = \sum_\mu [X^\mu] \,,
\end{equation}
where the sum is over all partitions $\mu \subset m \times k$ obtained
by adding a vertical strip of $p$ boxes to $\la$.

\begin{lemma}\label{lemma:cohom_present}
  The linear map
  $S/\langle \sigma_{k+1},\dots,\sigma_n \rangle \to H^*(X,\Q)$
  defined by $\sigma_\la \mapsto [X^\la]$ is an isomorphism of graded
  rings.
\end{lemma}
\begin{proof}
  Since $c_p = \sigma_{(1^p)}$ is mapped to $[X^{(1^p)}]$, it follows
  from \Fact{pieri} and the Pieri formula \eqn{pieri} that any product
  $c_p \sigma_\la$ is mapped to $[X^{(1^p)}] \cdot [X^\la]$.  Since
  $S$ is generated by $c_1,\dots,c_m$, this implies that the map is a
  ring homomorphism.
\end{proof}

\subsection{Quantum deformations of Grassmannians}%
\label{sec:qdeform_construct}

The quantum cohomology ring $\QH(X)$ contains a single deformation
parameter $q$ of degree $n$.  We next give a combinatorial
construction of the non-negative quantum deformations of $H^*(X,\Q)$.

Given any rational number $\al \in \Q$, define the ring
\[
  \QH_\al = S[q]/\langle \sigma_{k+1}, \dots, \sigma_{n-1}, \sigma_n +
  (-1)^m \al q \rangle \,.
\]
This ring is a graded $\Q[q]$-algebra, where the grading is defined by
$\deg(q) = n$ and $\deg(c_p) = p$ for $1 \leq p \leq m$.
\Lemma{cohom_present} implies that the assignments
$\sigma_\la \mapsto [X^\la]$ and $q \mapsto 0$ define a ring
homomorphism $\QH_\al \to H^*(X,\Q)$, and we show in
\Corollary{qh_basis} that $\QH_\al$ is a quantum deformation of
$H^*(X,\Q)$ with basis $\{ \sigma_\la \mid \la \subset m \times k \}$.
We first show that $\QH_\al$ satisfies the following rescaled version
of Bertram's quantum Pieri formula \cite{bertram:quantum}.

\begin{prop}\label{prop:qpieri}
  {\rm(a)}
  For $1 \leq p \leq m$ and $\la \subset m \times k$ we have in
  $\QH_\al$ that
  \[
    c_p\, \sigma_\la = \sum_\mu \sigma_\mu + \sum_\nu \al q\, \sigma_\nu \,.
  \]
  The first sum is over all partitions $\mu \subset m \times k$ that
  can be obtained by adding a vertical strip of $p$ boxes to $\la$.
  The second sum is empty unless $\la_1 = k$, in which case it is over
  all partitions $\nu$ obtained by removing $n-p$ boxes from the outer
  rim of $\la$, with at least one box removed from each
  column.\smallskip

  \noin{\rm(b)} For $1 \leq p \leq k$ and $\la \subset m \times k$ we
  have in $\QH_\al$ that
  \[
    \sigma_p\, \sigma_\la = \sum_\mu \sigma_\mu + \sum_\nu \al q\,
    \sigma_\nu \,.
  \]
  The first sum is over all partitions $\mu \subset m \times k$ that
  can be obtained by adding a horizontal strip of $p$ boxes to $\la$.
  The second sum is empty unless $\la_m \neq 0$, in which case it is
  over all partitions $\nu$ obtained by removing $n-p$ boxes from the
  outer rim of $\la$, with at least one box removed from each row.
\end{prop}
\begin{proof}
  It follows from \Fact{pieri} and \Lemma{Sbasis} that
  \[
    c_p\, \sigma_\la = \sum_\mu \sigma_\mu
  \]
  where the sum is over all partitions $\mu$ with at most $m$ rows
  that can be obtained by adding a vertical strip of size $p$ to
  $\la$.  If such a partition satisfies $\mu_1 \leq k$, then the term
  $\sigma_\mu$ is included in the first sum of part (a).  Otherwise we
  have $\la_1 = k$ and $\mu_1 = k+1$, in which case the first row of
  the $m \times m$ determinant defining $\sigma_\mu$ is equal to
  $(0,\dots,0,(-1)^{m-1} \al q)$ modulo the ideal defining $\QH_\al$.
  If $\mu_m = 0$, then the last row of the same determinant is
  $(0,\dots,0,1)$, in which case $\sigma_\mu = 0$.  On the other hand,
  if $\mu_m > 0$, then $\sigma_\mu = \al q\, \sigma_\nu \in \QH_\al$
  where $\nu = (\mu_2-1,\dots,\mu_m-1,0)$.  The second sum of part (a)
  contains exactly the terms obtained in this way.

  Part (b) is equivalent to part (a) with $m$ and $k$ interchanged.
  To make this explicit, let $c' = (c'_1, \dots, c'_k, 0, 0, \dots)$
  be a sequence of $k$ variables, and set
  $\sigma'_p = \Delta_{(1^p)}(c')$ for $p \geq 0$ and
  $\sigma'_\mu = \Delta_\mu(\sigma')$ for $\mu \subset k \times m$.
  Let $\phi : \Q[c'_1,\dots,c'_k] \to \Q[c_1,\dots,c_m]$ be the ring
  homomorphism defined by $\phi(c'_j) = \sigma_j$ for
  $1 \leq j \leq k$.  For $1 \leq p < n$ we have
  $\phi(\sigma'_p) = \Delta_{(1^p)}(\sigma_1, \dots, \sigma_k, 0,
  \dots) \equiv \Delta_{(1^p)}(\sigma) = c_p$ modulo the ideal
  $J = \langle \sigma_{k+1},\dots,\sigma_{n-1}\rangle$.  This implies
  that
  $\phi(\sigma'_\mu) \equiv \Delta_\mu(c) = \Delta_{\mu'}(\sigma) =
  \sigma_{\mu'}$ (mod $J$) for $\mu \subset k \times m$.  We also have
  $\phi(\sigma'_n) = \Delta_{(1^n)}(\sigma_1, \dots, \sigma_k, 0,
  \dots) \equiv \Delta_{(1^n)}(\sigma) + (-1)^n \sigma_n = (-1)^n
  \sigma_n$ (mod $J$); this term appears because $\sigma_n$ is the
  upper-right entry of the determinant defining
  $\Delta_{(1^n)}(\sigma)$.  Define
  \[
    \QH'_\al = \Q[q,c'_1,\dots,c'_k]/\langle \sigma'_{m+1}, \dots,
    \sigma'_{n-1}, \sigma'_n + (-1)^k \al q\rangle \,.
  \]
  The observed properties of $\phi$ show that this map factors as a
  $\Q[q]$-algebra homomorphism $\phi : \QH'_\al \to \QH_\al$ given by
  $\phi(\sigma'_\mu) = \sigma_{\mu'}$ for $\mu \subset k \times m$.
  Part (b) of the proposition therefore follows from part (a) applied
  to $\QH'_\al$.
\end{proof}

\begin{cor}\label{cor:qh_basis}
  $\QH_\al$ is a free $\Q[q]$-module with basis
  $\{ \sigma_\la \mid \la \subset m \times k \}$.
\end{cor}
\begin{proof}
  It follows from \Proposition{qpieri}(a) that $\QH_\al$ is generated as
  a $\Q[q]$-module by the classes $\sigma_\la$ for
  $\la \subset m \times k$, and \Corollary{nzd} implies that $q$ is a
  non-zero divisor in $\QH_\al$.  Assume that
  $\sum_\la a_\la \sigma_\la = 0$ is a non-trivial relation in
  $\QH_\al$, where the sum is over $\la \subset m \times k$ and
  $a_\la \in \Q[q]$.  Since $q$ is a non-zero divisor, we may assume
  that $a_\la(0) \neq 0$ for some $\la$.  But then the relation
  $\sum_\la a_\la(0) \sigma_\la = 0$ in $\QH_\al/\langle q \rangle$
  contradicts \Lemma{Hbasis}.
\end{proof}

\subsection{The Seidel representation}

\begin{defn}
  Given a partition $\la \subset m \times k$, define the first
  \emph{Seidel shift} of $\la$ by
  \[
    \la\uparrow1 = \begin{cases}
      (\la_1+1, \la_2+1, \dots, \la_m+1) & \text{if $\la_1 < k$,}\\
      (\la_2,\la_3,\dots,\la_m,0) & \text{if $\la_1 = k$.}
    \end{cases}
  \]
  The $p$-th Seidel shift $\la\uparrow p$ is defined inductively by
  $\la\uparrow0 = \la$ and
  $\la\uparrow (p+1) = (\la\uparrow p)\uparrow1$ for $p \geq 0$.
\end{defn}

The first Seidel shift captures the following case of Bertram's
quantum Pieri formula (\Proposition{qpieri}):
\begin{equation}\label{eqn:c-seidel}
  c_m\, \sigma_\la = (\al q)^{\frac{m+|\la|-|\la\uparrow1|}{n}}
  \sigma_{\la\uparrow1} \,.
\end{equation}
Notice that Seidel shifts of the zero partition are given by
\begin{equation}\label{eqn:seidelzero}
  0\uparrow p = \begin{cases}
    (p^m) & \text{for $0 \leq p \leq k$,}\\
    (k^{n-p}) & \text{for $k \leq p \leq n$.}
  \end{cases}
\end{equation}
It follows that $c_m^n = (\al q)^m$ and
$(\al q)^m \sigma_\la = c_m^n \sigma_\la = (\al
q)^{\frac{mn+|\la|-|\la\uparrow n|}{n}} \sigma_{\la\uparrow n}$, so we
obtain $\la\uparrow n = \la$ for any partition
$\la \subset m \times k$.  For any integer $p \in \Z$ we can therefore
define $\la\uparrow p = \la\uparrow p'$ where $p' \geq 0$ is chosen so
that $p' \equiv p$ (mod $n$).  We write
$\la\downarrow p = \la\uparrow(-p)$ for convenience.  With this
notation we have
\begin{equation}\label{eqn:t-seidel}
  \sigma_k\, \sigma_\la = (\al q)^{\frac{k+|\la|-|\la\downarrow1|}{n}}
  \sigma_{\la\downarrow1}
\end{equation}
in $\QH_\al$ by \Proposition{qpieri}(b).

\begin{lemma}\label{lemma:seidelsum}
  For any partition $\la \subset m \times k$ we have
  $\sum_{p=0}^{n-1} |\la\uparrow p| = \frac{1}{2} kmn$.
\end{lemma}
\begin{proof}
  The lemma follows from \eqn{seidelzero} when $\la = 0$, so it
  suffices to show that
  $\sum_{p=0}^{n-1} |\mu\uparrow p| = \sum_{p=0}^{n-1} |\la\uparrow
  p|$ holds whenever the partition $\mu$ is obtained by adding a
  single box to $\la$, say in row $i$ and column $j$.  By replacing
  the partitions with $\la\uparrow(k-j+i-1)$ and
  $\mu\uparrow(k-j+i-1)$, we may assume that the added box is in the
  upper-right corner of the rectangle $m \times k$, so that
  $\la_1 = k-1$ and $\mu_1 = k$.  The identity then follows because
  $|\mu\downarrow p| = |\la\downarrow p| + 1$ for $0 \leq p \leq n-2$
  and $|\mu\uparrow1| = |\la\uparrow1| - n+1$.
\end{proof}

\begin{cor}\label{cor:seidelsize}
  Let $\la$ and $\mu$ be partitions contained in $m \times k$ such
  that $|\la| > |\mu|$.  Then there exists $p \in \Z$ for which
  $|\la\uparrow p| < |\mu\uparrow p|$.
\end{cor}

\subsection{Main results}

Our classification of non-negative quantum deformations of
Grassmannians is based on the following theorem.

\begin{thm}\label{thm:classify_qdeform}
  Let $\{ \tau_\la \mid \la \subset m \times k \}$ be any
  $\Q[q]$-basis of $\QH_\al$ such that {\rm(1)} $\tau_\la$ is
  homogeneous of degree $|\la|$, {\rm(2)}
  $\tau_\la - \sigma_\la \in \langle q \rangle$, and {\rm(3)} the
  structure constants of $\QH_\al$ with respect to the $\Q$-basis
  $\{ q^d \tau_\la \mid \la \subset m \times k, d \geq 0 \}$ are
  non-negative.  Then we have $\tau_\la = \sigma_\la$ for all $\la$.
\end{thm}
\begin{proof}
  For each $\la \subset m \times k$ it follows from (1) and (2) that
  we can write
  \begin{equation}\label{eqn:tau2sigma}
    \tau_\la = \sigma_\la + \sum_{|\mu|<|\la|} a_{\la,\mu}\,
    q^{\frac{|\la|-|\mu|}{n}}\, \sigma_\mu \,,
  \end{equation}
  where $a_{\la,\mu} \in \Q$ is non-zero only if $|\la|-|\mu|$ is a
  positive multiple of $n$.  We must show that $a_{\la,\mu} = 0$ for
  all $\la$ and $\mu$.  Since $c_m \tau_k = \al q$, we must have
  $\al \geq 0$ by (3).

  Assume first that $\al = 0$.  If the theorem is false, then let
  $\la \subset m \times k$ be minimal in the degree-lexicographic
  order such that $\tau_\la \neq \sigma_\la$.  Set $\ell = \ell(\la)$
  and let $\wh\la = (\la_1-1,\la_2-1,\dots,\la_\ell-1)$ be the
  partition obtained by removing the first column of $\la$.
  \Proposition{qpieri}(a) shows that
  \[
    c_\ell \tau_{\wh\la} =
    c_\ell \sigma_{\wh\la} =
    \sum_\nu \sigma_\nu \,,
  \]
  where the sum is over all partitions $\nu \subset m \times k$ that
  can be obtained by adding a vertical strip of $\ell$ boxes to
  $\wh\la$.  Since $\la$ is the maximal such partition in the
  degree-lexicographic order, we obtain
  \[
    c_\ell \tau_{\wh\la} = \sigma_\la + \sum_{\nu\neq\la} \sigma_\nu =
    \tau_\la - \sum_{|\mu|<|\la|} a_{\la,\mu}\,
    q^{\frac{|\la|-|\mu|}{n}}\, \tau_\mu + \sum_{\nu \neq \la}
    \tau_\nu \,.
  \]
  This shows that $-a_{\la,\mu}$ is a structure constant of the basis
  $\cT = \{ q^d \tau_\la \}$ for each $\mu$, so (3) implies that
  $a_{\la,\mu} \leq 0$.  Now choose $\nu$ with $a_{\la,\nu} < 0$ such
  that the dual partition $\nu^\vee = (k-\nu_m,\dots,k-\nu_1)$ is
  maximal in the degree-lexicographic order.  If $\mu$ is any
  partition for which $a_{\la,\mu} \neq 0$, then either $\nu^\vee$ is
  larger than $\mu^\vee$ in the lexicographic order, in which case it
  follows from \Proposition{qpieri}(b) and induction on $p$ that
  \[
    \sigma_{\nu^\vee_1} \cdots \sigma_{\nu^\vee_p}\, \sigma_\mu =
    \begin{cases}
      \sigma_{(k^p,\mu_1,\dots,\mu_{m-p})} &
      \text{if $\mu_i=\nu_i$ for $m-p<i\leq m$,}\\
      0 & \text{otherwise.}
    \end{cases}
  \]
  Otherwise we have $|\nu^\vee| > |\mu^\vee|$, in which case
  $|\nu^\vee| + |\mu| > mk$ and
  $\sigma_{\nu^\vee_1} \cdots \sigma_{\nu^\vee_m}\, \sigma_\mu = 0$.

  Applying this to equation \eqn{tau2sigma} we therefore obtain
  \[
    \tau_{\nu_1^\vee} \cdots \tau_{\nu_m^\vee} \tau_\la
    = a_{\la,\nu}\, q^{\frac{|\la|-|\nu|}{n}} \sigma_{(k^m)} \,.
  \]
  Finally, since the coefficient of $\tau_{(k^m)}$ in the
  $\cT$-expansion of $\sigma_{(k^m)}$ is equal to 1, we deduce that
  $a_{\la,\nu}$ is an iterated structure constant, so (3) implies that
  $a_{\la,\nu} \geq 0$, a contradiction.

  Assume now that $\al > 0$.  Then $c_m \tau_k = \al q$ is a non-zero
  multiple of $q$.  If $\la \subset m \times k$ is any partition with
  $\la_1 < k$, then since
  $[X^{(1^m)}]\cdot [X^\la] = [X^{\la\uparrow1}]$ holds in
  $H^*(X,\Q)$, it follows from \Lemma{basis2basis} that
  $c_m \tau_\la = \tau_{\la\uparrow1}$.  On the other hand, if
  $\la_1 = k$, then $[X^k] \cdot [X^{\la\uparrow1}] = [X^\la]$ holds
  in $H^*(X,\Q)$, so \Lemma{basis2basis} implies that
  $c_m \tau_\la = c_m \tau_k \tau_{\la\uparrow1} = \al q\,
  \tau_{\la\uparrow1}$.  This shows that multiplication by $c_m$ in
  the basis $\cT$ is determined by
  \begin{equation}\label{eqn:seidel-c-tau}
    c_m \tau_\la = (\al q)^{\frac{m+|\la|-|\la\uparrow1|}{n}}
    \tau_{\la\uparrow1} \,.
  \end{equation}

  If $\tau_\la \neq \sigma_\la$, then choose $\mu$ such that
  $a_{\la,\mu} \neq 0$.  By \Corollary{seidelsize} we can choose $p$
  such that $|\la\uparrow p| < |\mu\uparrow p|$.  Equation
  \eqn{seidel-c-tau} then implies that
  $c_m^p \tau_\la = (\al q)^d \tau_{\la\uparrow p}$ where
  $d = \frac{mp+|\la|-|\la\uparrow p|}{n}$, so $c_m^p \tau_\la$ is
  divisible by $q^d$.  However, equation \eqn{c-seidel} shows that the
  product of $c_m^p$ with the right hand side of \eqn{tau2sigma}
  contains the term
  \[
    a_{\la,\mu}\, q^{\frac{|\la|-|\mu|}{n}}\,
    (\al q)^{\frac{mp+|\mu|-|\mu\uparrow p|}{n}}
    \sigma_{\mu\uparrow p} \,,
  \]
  and since $|\la\uparrow p| < |\mu\uparrow p|$, this term is not
  divisible by $q^d$.  This contradicts that $c_m^p \tau_\la$ is
  divisible by $q^d$, which completes the proof.
\end{proof}

\begin{cor}\label{cor:classify_qdeform}
  Let $X = \Gr(m,n)$ and let $\QH_\al = (\QH_\al,\{\sigma_\la\})$ be
  the quantum deformation of $H^*(X,\Q)$ defined
  in \Section{qdeform_construct}.\smallskip

  \noin{\rm(a)} The deformation $\QH_\al$ is non-negative if and only
  if $\al \geq 0$.\smallskip

  \noin{\rm(b)} Any non-negative quantum deformation of $H^*(X,\Q)$ is
  isomorphic to $\QH_\al$ for a unique rational number
  $\al \geq 0$.\smallskip

  \noin{\rm(c)} $\QH(X) \cong \QH_1$ is the only non-negative
  change-of-basis quantum deformation of $H^*(X,\Q)$.
\end{cor}
\begin{proof}
  Given any quantum deformation $(\QH, \{\tau_\la\})$ of $H^*(X,\Q)$,
  it follows from \Definition{qdeform} that
  $\tau_{(1^m)} \tau_k = \al q$ for some $\al \in \Q$.  Define a ring
  homomorphism $\phi : S[q] \to \QH$ by $\phi(c_p) = \tau_{(1^p)}$ and
  $\phi(q) = q$.  \Definition{qdeform} then implies that
  $\phi(\sigma_p) = \tau_p$ for $1 \leq p \leq k$,
  $\phi(\sigma_p) = 0$ for $k < p < n$, and
  $\phi(\sigma_n) = \phi\left(\sum_{p=1}^m (-1)^{p-1} c_p
    \sigma_{n-p}\right) = (-1)^{m-1} \tau_{(1^m)} \tau_k = (-1)^{m-1}
  \al q$.  We deduce that $\phi$ factors as an isomorphism of
  $\Q[q]$-algebras $\psi: \QH_\al \to \QH$, and the set
  $\{ \psi^{-1}(\tau_\la) \}$ is a deformation of the basis
  $\{ \sigma_\la \}$ of $\QH_\al$.  If $\QH$ is a non-negative quantum
  deformation, then \Theorem{classify_qdeform} implies that
  $\psi^{-1}(\tau_\la) = \sigma_\la$ for each $\la$.  This proves part
  (b).

  Part (c) follows from part (b) by using that $\QH(X)$ is a
  non-negative quantum deformation of $H^*(X,\Q)$ satisfying the
  relation $[X^{(1^m)}] \star [X^k] = q$.  This relation is equivalent
  to Witten's presentation of $\QH(X)$ \cite[\S 3.2]{witten:verlinde}
  (see also \cite[Prop.~3.1]{siebert.tian:quantum}).  The uniqueness
  statement of part (c) follows from \Theorem{classify_qdeform}
  applied to $\QH_1$.

  For $\al > 0$ the assignments $q \mapsto \al q$ and
  $c_p \mapsto c_p$ define an isomorphism of graded rings
  $\QH_1 \cong \QH_\al$.  Since this map sends $\sigma_\la$ to
  $\sigma_\la$, it follows that the structure constants of $\QH_\al$
  are the rescaled Gromov-Witten invariants
  $\al^d \gw{[X^\la],[X^\mu],[X_\nu]}{d}$.  This shows that $\QH_\al$
  is non-negative for $\al > 0$.  The trivial deformation
  $\QH_0 = H^*(X,\Q) \otimes_\Q \Q[q]$ is also non-negative, and the
  relation $c_m \sigma_k = \al q$ shows that $\QH_\al$ is not
  non-negative for $\al < 0$.  This proves part (a).
\end{proof}

\begin{cor}\label{bertram}
  The quantum cohomology ring $\QH(X)$ satisfies Bertram's quantum
  Giambelli formula
  \[
    [X^\la] = \det \left( [X^{\la_i+j-i}] \right)_{m \times m}
  \]
  and quantum Pieri formula
  \[
    [X^p] \star [X^\la] = \sum_\mu [X^\mu] + \sum_\nu q\, [X^\nu] \,,
  \]
  where the sums are over the same partitions as in
  \Proposition{qpieri}(b).
\end{cor}
\begin{proof}
  This follows from \Corollary{classify_qdeform},
  \Proposition{qpieri}, and the definition of $\sigma_\la$.
\end{proof}

\begin{example}\label{example:counterexample}
  Let $X = \GL(6)/B$ be the variety of complete flags in $\C^6$ and
  set $t = (6,1,2,3,4,5)$.  We have $\pi_1(\Aut(X)) \cong \Z/6\Z$, and
  the Seidel representation of this group on
  $\QH(X)/\langle q-1 \rangle$ is defined by quantum multiplication
  with $[X^t]$, see \cite{seidel:1, chaput.manivel.ea:affine}.
  Products with this class modulo $\langle q-1 \rangle$ are given by
  $[X^t] \star [X^w] = [X^{tw}]$ \cite{ciocan-fontanine:quantum*2}.
  The following table shows the Seidel orbits of the Schubert classes
  given by $e = (1,2,3,4,5,6)$ and $v = (3,2,1,6,5,4)$, using the
  notation $w \uparrow r = t^r w$.\medskip

  \begin{center}
    \begin{tabular}{|c|c|c|c|c|}
      \hline
      $r$ & $e\uparrow r$ & $\ell(e\uparrow r)$
      & $v\uparrow r$ & $\ell(v\uparrow r)$ \\
      \hline
      0 & 123456 & 0 & 321654 & 6 \\
      1 & 612345 & 5 & 216543 & 7 \\
      2 & 561234 & 8 & 165432 & 10 \\
      3 & 456123 & 9 & 654321 & 15 \\
      4 & 345612 & 8 & 543216 & 10 \\
      5 & 234561 & 5 & 432165 & 7 \\
      \hline
    \end{tabular}
  \end{center}
  \medskip

  \noin This shows that the analogues of \Lemma{seidelsum} and
  \Corollary{seidelsize} for varieties of complete flags are false.
\end{example}

\providecommand{\bysame}{\leavevmode\hbox to3em{\hrulefill}\thinspace}
\providecommand{\MR}{\relax\ifhmode\unskip\space\fi MR }
\providecommand{\MRhref}[2]{%
  \href{http://www.ams.org/mathscinet-getitem?mr=#1}{#2}
}
\providecommand{\href}[2]{#2}

\bibliographystyle{amsplain}

\end{document}